\documentclass{amsart}
 \usepackage{amsthm,amsfonts,amsmath,amssymb,latexsym,epsfig,upref,eucal,ae}
\usepackage[all]{xy}

\usepackage{graphicx}

\usepackage{subfigure}
\usepackage{epic}
\usepackage{eepic}
\usepackage{setspace}
\usepackage{booktabs}
\usepackage{longtable}
\usepackage{pstricks}
\usepackage{url}

\newtheorem{theorem}[subsubsection]{Theorem}
\newtheorem{lemma}[subsubsection]{Lemma}
\newtheorem{proposition}[subsubsection]{Proposition}

\newenvironment{remark}{\medskip \refstepcounter{subsubsection}
\noindent  {\bf Remark \thetheorem}.\rm}{\,}
\newenvironment{definition}{\medskip \refstepcounter{subsubsection}
\noindent  {\bf Definition \thetheorem}.\rm}{\,}

\def\P{\mathbb{P}}

\def\grad{\mathrm{grad}}

\def\Aut{\mathrm{Aut}}

\def\RR{\mathbb{R}}
\def\CC{\mathbb{C}}
\def\ZZ{\mathbb{Z}}
\def\CP{\mathbb{CP}}
\def\TT{\mathbb{T}}

\def\cO{\mathcal{O}}

\def\cJ{\mathcal{J}}
\def\cG{\mathcal{G}}

\def\tsheaf{\Theta}

\def\cL{\mathfrak{L}}
\def\del{\partial}
\def\delb{\overline\partial}
\def\cK{\mathcal{K}}

\def\h{\mathfrak{h}}

\def\g{\mathfrak{g}}

\def\R{\mathbb{R}}
\def\Q{\mathbb{Q}}
\def\Z{\mathbb{Z}}
\def\N{\mathbb{N}}
\def\P{\mathbb{P}}
\def\F{\mathbb{F}}

\def\C{\mathbb{C}}

\def\cG{\mathcal{G}}
\def\cJ{\mathcal{J}}
\def\cX{\mathcal{X}}
\def\Xbar{\overline{X}}
\def\Xhat{\widehat{X}}

\def\Om{\Omega}
\def\om{\omega}

\begin{document}

\title[Deformations of extremal toric manifolds]
{Deformations of extremal toric manifolds}
\author[Y. Rollin]{Yann Rollin}
\author[C. Tipler]{Carl Tipler}
\address{D\'epartment de Math\'ematiques, Laboratoire Jean Leray,
2, Rue de la Houssini\`ere - BP 92208, F-44322 Nantes, FRANCE}
\email{yann.rollin@univ-nantes.fr, 
carl.tipler@univ-nantes.fr}

\begin{abstract}
Let $X$ be a compact toric extremal K\"ahler manifold.
Using the work of Sz\'ekelyhidi \cite{sz},
we provide a combinatorial criterion on the fan describing $X$
to ensure the existence of complex deformations of $X$
that carry extremal metrics. As an example, we find new CSC
metrics on $4$-points blow-ups of $\C\P^1\times\C\P^1$.
\end{abstract}

\maketitle

\section{Introduction}
\label{secintro}
Existence of extremal K\"ahler metrics is a very hard problem initiated
by Calabi which has
been solved for some special cases. 
More precisely, given a complex manifold $X$ together with an ample
line bundle $L\to X$, we are looking for an extremal metric with
K\"ahler class $c_1(L)$.
The conjecture of Donaldson Tian and Yau,
refined by Sz\'ekelyhidi in the  extremal case, is that the existence of
such an extremal metric should be equivalent to the (relative)
K-polystability of $(X,L)$ (or some refinement of this notion).

In the case where $X$ is  a toric surface, the problem has been
completely solved by Donaldson: in this case if the Futaki invariant
vanishes, the existence of a constant scalar curvature K\"ahler metric
is equivalent to the $K$-polystability of $(X,L)$ (cf.  \cite{Don1}, \cite{Don2},\cite {Don3}).

Motivated by this result, we would like to study existence of extremal
metrics on complex surfaces with complex structure close to a toric 
complex surface carrying an extremal metric. 
The main tool to achieve this goal is the deformation
theory of constant scalar curvature Kähler (in short CSCK) metrics developed by Sz\'ekelyhidi \cite{sz}, and generalized
by Br\"onnle \cite{bro} in the extremal case.
Roughly, the idea is that small complex deformations which are stable in the GIT sense are the one carrying CSCK metrics.
In the case of toric manifolds, the space of complex deformations is described in a combinatorial way, using the fan that defines the toric variety.
This picture is particularly clear thanks to the theory of T-varieties due to Altmann, Ilten and Vollmert. Relying on \cite{ilv},
the stable deformations can be determined explicitly.

We should point out that the perturbation technique used to construct extremal
metrics is particularly nice, since it leads to a local
version of the
Donaldson-Tian-Yau conjecture:
let  $X\hookrightarrow\cX\to B$ be a family of complex deformations of $X\simeq
\cX_0$, where $B$ is an open neighborhood of the origin in some
complex vector space. Let  $\cL\to \cX$ be a polarization
of the deformation, that is a holomorphic line
bundle such that the restriction $\cL_t\to \cX_t$ is ample for all
$t\in B$. Assume that $\Omega=c_1(\cL_0)$ is represented by the K\"ahler class
of an extremal metric $\omega_0$ on $\cX_0$. Let $H$ be 
 the  group of Hamiltonian isometries of $\omega_0$ and
 $G\subset H$ be a compact connected Lie groups acting
holomorphically on $\cX$ and fixing the fibers of $\cX\to B$.
We are also assuming that the Lie algebra of $G$ contains the extremal vector field of
$\omega_0$. The latter condition is equivalent to the vanishing of the
reduced scalar curvature $s^G_{\omega_0}$.
Then up to the cost of shrinking  $B$ to a sufficiently small neighborhood
  of the origin, we have the following property in the case where $G$
  is a torus :
for every $t\in B$ such that $\cL_t\to\cX_t$ is K-polystable relative
to $G$, the complex manifold  $\cX_t$ carries an extremal metric.
In the case where  $G$ is a maximal torus of $\Aut(\cX_t)$,  the condition of K-polystability is also
  necessary  by a result of Stoppa-Sz\'ekelyhidi \cite{ss09}. In the
  case where $G$ is a torus but is not maximal in $\Aut(\cX_t)$, the
  latter statement is not clear although it would be reasonable to
  expect such a result (cf. \S\ref{sec:polarized} for details).

\subsection{A typical example}
\label{sec:exintro}
Before stating general results, we would like to start with a nice and
simple example (see section \ref{sec:desccompat} for a detailed study). Here we have a CSCK surface which admits complex deformations of different types.
Some of them are CSCK whereas others do not admit any CSCK metric. This is closely related to the behaviour of the Mukai-Umemura
$3$-fold and its deformations \cite{Don08} and we believe that the general theory should benefit from the study of such situations.

We endow $\CP^1\times\CP^1$ with a CSCK metric deduced from a
product of metrics of constant curvature metrics on each factors. Then
we get a CSCK orbifold $\Xbar=(\CP^1\times \CP^1)/\ZZ_2$ where the
action of $\ZZ_2$ is generated by a rotation of order $2$ of each
factor. The minimal resolution $\Xhat\to \Xbar$ is a $4$-points
blow-up of $\CP^1\times\CP^1$. More concretely, let $p_+=[0:1]$ and
$p_-=[1:0]$ be two points on $\CP^1$. The points $p_\pm$  are fixed under the
$\CC^*$-action defined by $\lambda\cdot [x:y]= [\lambda x:y]$. We
deduce a toric action on $\CP^1\times \CP^1$ with four fixed points
$$P_0=(p_+,p_+), P_1 = (p_+,p_-), P_1'=(p_-,p_+) \text{ and } P_\infty=(p_-,p_-).$$
Blowing up the fixed points $P_j$, we obtain the resolution $\Xhat$, with
the induced toric action.

It is known that
$\Xhat$ carries a CSCK metric $\omega$ with K\"ahler class denoted
$\Omega\in H^2(\Xhat,\RR)$ (cf. \cite{rs}) and we are trying to
understand which small complex deformations of $\Xhat$ are also extremal.
In addition, the CSCK metric can be chosen to have integral K\"ahler
class $\Omega$.

Let $\Xhat \hookrightarrow \cX\to B$ be a
toric semiuniversal family of deformations (cf. Definition
\ref{def:semiheq}) of $\Xhat$ as above. 
Here $B$ is identified to a neighborhood of the origin
 in 
$H^1(\Xhat,\tsheaf_{\Xhat})$ where $\tsheaf_{\Xhat}$ denotes the
tangent sheaf to $\Xhat$. 
Then  $H^1(\Xhat,\tsheaf_{\Xhat})$ admits a basis  $(e_1,e_2,e_3,e_4) $
such that the complex deformation associated to $(x_1,x_2,x_3,x_4)\in
B$ corresponds to moving  the blown-up points $P_0$ and $P_\infty$ 
given by the coordinates $P_0(x_1,x_3)=([x_1:1],[x_3:1])$ and $P_\infty(x_2,x_4)=([1:x_2]\times[1:x_4])$. 
The deformation $\cX$ is endowed with a natural action of the
real torus.
In this basis, the toric action of $(\lambda,\mu)\in \TT^\CC$ is represented by the matrix
$$
\left[ 
\begin{array}{cccc}
                                          \lambda & 0 & 0 & 0 \\
                                            0  & \lambda^{-1} & 0 & 0 \\
                                             0 & 0 & \mu & 0 \\
                                               0 & 0 & 0 & \mu^{-1}
                                         \end{array} 
\right].
$$
Thus, using the induced isomorphism  $H^1(\Xhat,\tsheaf_{\Xhat})\simeq
\CC^4$ the set of polystable points under the toric action in the GIT
sense is given by  
\begin{equation}
  \label{eq:U}
U= 
U_0\cup U_2'\cup U_2''\cup U_4  
\end{equation}
where 
\begin{align*}
  U_0 &=\{0\}\\
U'_2 &=  \{ (x_1,x_2,0,0)\in \C^4, x_1x_2\neq 0\} \\
U''_2 &=  \{ (0,0, x_3,x_4)\in \C^4, x_3x_4\neq 0\}\\
U_4 &=  \{ (x_1,x_2,x_3,x_4)\in \C^4, x_1x_2x_3x_4\neq 0\}
\end{align*}
We should point out that the only toric variety is $\cX_0$ whereas
$\cX_t$ admits only a residual $\CC^*$-action for $t\in U'_2\cup
U''_2$ and no holomorphic holomorphic vector field if $t\in U_4$.
 Then we have the following result:
\begin{theorem}
Let $\Xhat$ be the complex surface described above,  $\Omega$  the K\"ahler class of a constant scalar curvature
K\"ahler metric on $\Xhat$, and
 $\Xhat\hookrightarrow\cX\to B$,  a semiuniversal toric family of deformations of $\Xhat$,
 where $B$ is an open neighborhood of  the origin in
 $H^1(\Xhat,\tsheaf_{\Xhat})$. 

  Up to the cost of shrinking $B$ to a smaller open neighborhood of
  the origin in
  $H^1(\Xhat,\tsheaf_{\Xhat})$, we have the following property:
For $t\in B$, the complex surface $\cX_t$ admits a K\"ahler  metric of
constant scalar curvature with K\"ahler class $\Omega$ if $t\in U $,
where $U$ is the set of polystable points described at \eqref{eq:U}.
If $\Omega$ is an integral cohomology class, this condition is also necessary.
\end{theorem}

\begin{remark}
If $\Omega$ is integral, the special deformations given by $B \setminus U$ do
 not carry extremal metrics with K\"ahler class $\Omega$. They play a
 role analogue to
 the famous family of deformations of the Mukai-Umemura $3$-fold given
 by Tian \cite{Ti}.
\end{remark}

\subsection{General results}
The $K$-polystability condition is generally very hard to check. In
fact, there is a much simpler condition that we shall use in practice.
Let $X$ be a toric manifold such that
the real torus $\TT$ is a maximal  connected compact subgroup of
$\Aut(X)$ and $H^2(X,\tsheaf_X)=0$. Then by Kodaira-Spencer theorem $X$ admits a
 semiuniveral toric family of complex deformations
$X\hookrightarrow \cX\to B$ (cf. Definition \ref{def:semiheq}).

Let $\Omega$ be a K\"ahler class represented by the K\"ahler form of a
CSCK metric on $X$.
As $X$ is toric we have $h^{0,2}(X)=0$ hence $\Omega$ belongs to the K\"ahler cone of $\cX_t$
for all $t\in B$ sufficiently small. The question is now whether there
exists a CSCK metric on $\cX_t$ with K\"ahler class $\Omega$.

 The complex torus
$\TT^\CC$ acts naturally on $H^1(X,\tsheaf_X)$ and it follows from 
 Sz\'ekelyhidi's results that 
 $\cX_t$ carries a CSCK metric with K\"ahler class $\Omega$
if   $t\in B$  is sufficiently small and belongs to a polystable
orbit of $H^1(X,\tsheaf_X)$ under the  $\TT^\CC$-action.

In the toric case, $H^1(X,\tsheaf_X)$ is easily described in terms of the fan
defining the toric manifold. Moreover the torus action is also
explicit and the weights are readily computed. It follows that we have
an easy numerical criterion to characterize polystable orbits
as explained below (see \S\ref{stable} for the proof).

Let $\Sigma$ be the fan describing $X$ in a lattice $N$ and let $N^*$ denote the dual of the lattice $N$. 
Let $\Sigma^{(1)}$ be the set of rays in $\Sigma$, identified with primitive generators
of these rays. Following Ilten and Vollmert \cite{ilv}, we can compute from the fan $\Sigma$ a finite set
$N_{def}^*(\Sigma)\subset N^*$ (cf. \S\ref{stable}) which is the set of weights of the torus action on $H^1(X,\tsheaf_ X)$.
Then $H^1(X,\tsheaf_ X)$ admits a decomposition of the form:
$$
H^1(X,\tsheaf_ X)=\bigoplus_{R\in N^*_{def}(\Sigma)} H^1(X,\tsheaf_ X)(R).
$$
We proceed now with some definitions in order to state our main results.
 We say that a nonempty finite family $R_{1},..,R_{r} \in N^*$ is
 \textit{balanced}
if there exist positive integers
$a_1,\cdots,a_r$ such that $  a_1R_1+\cdots + a_r R_r =0$.
For each $R\in N^*\setminus  0 $, we introduce the sets
\begin{align*}
\lbrace R<0 \rbrace &=\lbrace x\in N |  \langle R, x \rangle <0 \rbrace 
\end{align*}
and
\begin{align*}
\lbrace R=0 \rbrace &=\lbrace x\in N |  \langle R, x \rangle =0 \rbrace.
\end{align*}
As $N^*_{def}$ is finite, we shall use the notation
$N^*_{def}=\{R_1,\cdots,R_s\}$. 
Then let  $\mu(\Sigma)$ be the set of all multi-indices
$I\subset \lbrace 1,..,s \rbrace$, such that
\begin{enumerate}
\item there exists a subfamily $J\subset I$ such that $\{R_{j}, j\in J\}$
is balanced and
\item $
\displaystyle N = \left (\bigcup_{i\in I} \lbrace R_i<0 \rbrace \right ) \cup \left
  ( \bigcap_{i\in I}  \lbrace R_i=0 \rbrace \right ).
$
\end{enumerate}
\begin{remark}
  Condition $(2)$ is automatically satisfied if $\{R_i,i\in I\}$ is
  balanced. Therefore $I\in \mu(\Sigma)$ and it follows that
  $\mu(\Sigma)\neq \emptyset$ in this case. 
\end{remark}

\noindent For each family of indices $I\subset \{1,\cdots,s\}$, we consider the direct sum $$W_I=\bigoplus
_{i\in I} H^1(X,\tsheaf_ X)(R_i).$$
 Each vector $x\in W_I$ is written  $x=\sum_{i\in I}
x_i$  with $x_i\in H^1(X,\tsheaf_ X)(R_i)$.
Let $V_I\subset W_I$ be the finite union of
subvector spaces given by the equations $x_i=0$ for some $i\in I$.
Put
$$
S_I=W_I \setminus V_I,
$$
Then the set of polystable points $H^1(X,\tsheaf_ X)^{p}$ is given by
the following proposition:

\begin{proposition}
Let $X$ be a smooth compact toric manifold given by a fan $\Sigma$ in a
lattice $N$.
Then, the set of polystable points of  $H^1(X,\tsheaf_ X)$ under the toric
action is given by
$$
H^1(X,\tsheaf_ X)^{p} =\{0\} \cup \bigcup_{I\in \mu(\Sigma)} S_I.
$$
In particular $H^1(X,\tsheaf_ X)^{p}\setminus\{0\}$ is not empty 
if and only if there is a balanced  family  in $N^*_{def}(\Sigma)$.
\end{proposition}
As an application  we obtain the following result:
\begin{theorem}
\label{equivCSCi}
Let $X$ be a smooth compact toric manifold defined by a fan $\Sigma$
in a lattice $N$ and let $g$ be a K\"ahler metric of constant scalar
curvature on $X$, with K\"ahler class $\Omega$, such that   its group of Hamiltonian isometries  $H$ 
satisfies $H^{\C}=\TT^{\C}$. 

Assuming that $H^2(X,\tsheaf_
X)=0$ we consider  the semiuniversal toric family of deformations
$X\hookrightarrow \cX\to
B$ of
$X\simeq \cX_0$, with $B$ identified to an open neighborhood of 
 the
origin in $H^1(X,\tsheaf_
X)$.

Then, up to the cost of shrinking $B$ to a sufficiently small open neighborhood
of the origin, the deformation $\cX_t$ for $t\in B\setminus 0$ admits a K\"ahler
metric of constant scalar curvature with K\"ahler class $\Omega$ if 
 $t\in  \bigcup_{I\in \mu(\Sigma)} S_I$. This condition is also
 necessary if the K\"ahler class $\Omega$ is integral.

In particular, $X$ admits non trivial complex deformations endowed with
constant scalar curvature K\"ahler metrics representing the K\"ahler class $\Omega$
if there is a balanced family in
$N^*_{def}(\Sigma)$.
\end{theorem}
We also have a more general result which can be used to deform
extremal metrics. In this case we have to work with complex
deformations preserving the extremal vector field and the stability is
replaced by a condition of relative stability modulo a subtorus that
contains the extremal vector field~(cf. \S\ref{equivextremal}).

\begin{remark}
The case of surfaces deserves special attention as in that case
$H^2(X,\tsheaf_ X)=0$ and a simple combinatorial criterion on the fan
ensures that $H^{\C}=\TT^{\C}$. It follows that the previous theorem easily
provides numerous examples of new extremal  metrics on deformations of
toric
 surfaces.
\end{remark}

\subsection{Plan of the paper}
The deformation theory of extremal metrics following Sz\'ekelyhidi and Br\"onnle is recalled at Section~\ref{deformation}. 
Section~\ref{toric} is devoted to investigate the stability criterion for toric manifolds and
in the last section we provide applications.

\subsection{Acknowledgments}
The authors would like to thank Simon Donaldson and Gabor Sz\'ekelyhidi who kindly
answered our questions. We are most indebted to Paul Gauduchon who
provided us with a copy of his excellent upcoming book \cite{gbook}.  We also thank Till Br\"onnle for communicating his thesis,
Nathan Ilten for sharing his insights on T-varieties as well as Yalong Shi and
Haozhao Li for some useful remarks. The authors would also like to
thank the referee for helpful comments.

\section{Deformations of extremal metrics}
\label{deformation}
Let $X$ be a compact K\"ahler manifold with K\"ahler metric $g$,
K\"ahler form $\om$ and
complex dimension $n$. 
The complex manifold $X$ is understood as a pair $(M,J_0)$ where 
$M$ is the underlying differentiable manifold  and $J_0$ is the
integrable almost complex structure.
From now on, the metric $g$ is assumed to be  extremal, or equivalently, 
$J_0\grad_g(s_g)$ is a Killing field. 

Let $\cJ$ be the space of almost complex structures on $M$ compatible
with  $\om$. The scalar curvature can be interpreted as a moment map
for the action on $\cJ$ induced by Hamiltonian transformations
of $(M,\omega)$~\cite{fuj,Don97}. This beautiful formal picture can be used
in practice to study deformations of extremal metrics (with fixed
K\"ahler class) in the CSCK case \cite{sz}. 

We show in the next sections how Székelyhidi's approach extends in the case
of extremal metrics of non constant scalar curvature. A similar
extension appears in an independent work of Br\"onnle \cite{bro} with a
rather different presentation and results. 

Although such a generalization seems very natural to experts in the
field, we shall give an account of the objects and techniques used for
the sake of completeness.

\subsection{The relative moment map}
In order to obtain smooth deformations of the extremal metric $g$, it is necessary to preserve the action of the extremal vector field $J_0\grad_g(s_g)$.

Let $H$ be the compact group of 
Hamiltonian isometries of $(M,J_0,\om)$ and $\cG$ the group of Hamiltonian transformations of $(M,\omega)$. The extremal
vector field generates an action by isometries on $(M,J_0,\om)$ corresponding to a subgroup $H_s$ of $H$. 
Let $G$ be a connected compact subgroup of $H$ containing $H_s$
as a subgroup, and $\cJ^{G}$ the subspace of $\cJ$ of almost complex structures
that are $G$-invariant. Denote by $\cG_G$ the normalizer of
$G$ in $\cG$. Then $\cK=\cG_G/G$ acts on $\cJ^{G}$.
The space of momenta, including constants, of elements of
the Lie algebra $\g$ of $G$ is denoted  $P^G_{\om}$ . If $v\in \h$, the Lie algebra of $H$, we choose the momenta $f_v$  to satisfy
$$
 -df_v=\om(v,\cdot)
$$
and the normalization
$$
\int_M f_v \om^n = 0,
$$
and we set $v=v_f$.
Then define $\Pi_{\om}^G$ to be the $L^2$-orthogonal projection
from $C^{\infty}(M)$ onto $P_{\om}^G$ induced by the volume form $\omega^n$.

\begin{definition}
 Let $J\in \cJ^G$. The reduced Hermitian scalar curvature $s_J^G$ of $(M,J,\om)$
is defined by
$$
s_J^G=s_J - \Pi_{\om}^G(s_J),
$$
where $s_J$ is the Hermitian scalar curvature of the almost K\"ahler
metric defined by $\omega$ and $J$.
\end{definition}

Using the Hamiltonian construction, the Lie algebra of $\cK$ is identified with the space
of $G$-invariant functions of mean value zero $C^{\infty}_0(M)^G$. This space is isomorphic to its dual via the $L^2$
inner product induced by $\om$.
Then the following is a generalization of Fujiki's \cite{fuj} and Donaldson's \cite{Don97} work 

\begin{proposition}[\cite{gbook}]
The space $\cJ^G$ admits a K\"ahler structure such that
 the action of $\cK$ on $\cJ^G$ is Hamiltonian and its moment map is given by
$$
\begin{array}{cccc}
 & \cJ^G & \rightarrow & C_0^{\infty}(M)^G \\
 & J & \mapsto & s^G_J
\end{array}
$$
\end{proposition}

It is a hard problem in general to find zeros of this moment map. We will show in the next section that if we start from a zero,
then looking for nearby zeros can be reduced to a finite dimensional problem.

\subsection{Reduction to finite dimensional stability}
\label{reduction}
The supscript $G$ will stand for $G$-invariant tensors.
Define an infinitesimal complexified action of $\cK$ on $\cJ^G$ by 
$$ \begin{array}{cccc}
P_J : &C_0^{\infty}(M,\C)^G &\rightarrow &T_J\cJ ^G\\
 & h & \mapsto & \delb v_h\\
\end{array}$$
We say that  $J_0$ and $J_1$ lie in the same "$\cK^{\C}$-orbit" if there is a path $\phi_t\in C^{\infty}_0(M,\C)^G$ 
and a path $J_t$ in $\cJ^G$ joining $J_0$ and $J_1$ such that
$$
\frac{d}{dt}J_t=P_{J_t}(\phi_t).
$$
Together with the operator
$$
\delb: T_{J_0}\cJ^G \rightarrow \Om^{0,2}(T^{1,0})^G
$$
The elliptic complex
$$
C^{\infty}_0(M,\C)^G  \rightarrow T_{J_0}\cJ^G \rightarrow \Om^{0,2}(T^{1,0})^G
$$
defines a finite dimensional vector space
$$
\widetilde{H}^1_G=\lbrace \alpha \in T_{J_0}\cJ^G / P_{J_0}^*\alpha=0 , \delb\alpha =0 \rbrace
$$
which is the kernel of the elliptic operator $P_{J_0}P_{J_0}^*+(\delb^*\delb)^2$ on $T_{J_0}\cJ^G$.
If $G=0$, this space parametrizes infinitesimal complex deformations of $(M,J_0)$ that are compatible with $\om$,
up to exact symplectomorphisms. When $G$ is not trivial, we obtain infinitesimal deformations that preserve $G$.

The hypothesis $H_s\subset G$ is re-interpreted as
$$s_{J_0}^G=0.$$

Let $H_G$ be the normalizer of $G$ in $H$ and $K_G=H_G/G$. Then $K_G$, and its complexification $K_G^{\C}$, act on $\widetilde{H}^1_G$. Let $\h_G$ denote the Lie algebra of $H_G$.
Then following \cite{sz} and working $G$-invariantly
we can state:

\begin{proposition}
\label{slice}
 There exists an open neighborhood of the origin $B_G\subset
 \widetilde{H}^1_G$  and a $K_G$-equivariant  map
$\Phi^G: B_G \rightarrow \cJ^G$ such that the $\cK^{\C}$-orbit of every integrable
complex structure $J\in \cJ^G$ near $J_0$ intersects the image of
$\Phi^G$. If $x,x'$ are in the same $K_G^{\C}$-orbit and
$\Phi^G(x)$ is integrable 
then $\Phi^G(x)$ and $\Phi^G(x')$ are in the same $\cK^{\C}$-orbit.
Moreover, for all $x\in B_G$ we have $s^G(\Phi(x))\in \h_G/\g$.
\end{proposition}

A moment map for the $K_G$-action on $B_G$ with respect to the pulled back K\"ahler form by $\Phi^G$ is
$$
\mu^G(x)=s^G(\Phi^G(x)).
$$
As points in $\widetilde{H}^1_G$ in the same $K_G^{\C}$-orbit correspond via 
$\Phi^G$ to points in the same $\cK^{\C}$-orbit if they represent integrable complex structures,
the problem of finding zeros for the moment map $s^G$ is reduced to the problem
of finding zeros of $\mu^G$ in $B_G$.
Using the Kempf-Ness theorem on $\widetilde{H}^1_G$ with the linear 
symplectic form induced by $\mu^G$, we obtain from \cite{sz}:

\begin{proposition}
\label{linearstability}
 After possibly shrinking $B_G$, suppose that $x\in B_G$ is
polystable for the linear $K_G^{\C}$ action on $\widetilde{H}^1_G$.
Then there exists $x'\in B$ in the $K_G^{\C}$-orbit of $x$ such that
$\mu^G(x')=0$.
\end{proposition}

The proof of this proposition relies on general properties for moment maps and extends directly to the $G$-invariant context.
The following theorem is an application of Proposition \ref{linearstability} :
\begin{theorem}
Let $J_0\in \cJ^G$ be an integrable complex structure such that the
corresponding metric satisfies $s^G_{J_0}=0$. 
  Let $B_G\subset \tilde H^1_G$ and $\Phi^G:B_G\to \cJ^G$  be an
  adapted slice (cf. Proposition \ref{slice}) with $B_G$ a
  sufficiently small neighborhood of the origin.
Then for every  polystable orbit $\cO\subset  H^1_G$ relative
to the linearized  action of $K_G^\CC$ on $\tilde H^1_G$,
the intersection $\cO\cap B_G$ is either empty or contains a unique point $t$ such that the metric deduced from $J_t$
satisfies $s^G_{J_t}=0$. 
\end{theorem}

\subsection{The polarized case}
\label{sec:polarized}

Let $(X,L)$ be a polarized extremal K\"ahler manifold and $(X',L')$ a small complex deformation of $(X,L)$.
Then Sz\'ekelyhidi has shown in the CSC case that the K-polystability of $(X',L')$ implies
the stability of the corresponding infinitesimal deformation.
This results holds in the extremal case.

\begin{theorem}
\label{extremaldefo}
Let $(X,L)$ be a polarized extremal K\"ahler 
manifold with extremal metric $\om$.

Let $G$ be a torus of $H$ the group of Hamiltonian isometries of $(X,\om)$ such that 
its Lie algebra contains the extremal vector field
and $\cL \rightarrow \cX\rightarrow \mathcal{B}$ a $G$-invariant polarized deformation of $(X,L)$.

Then, shrinking $\mathcal{B}$ if necessary,
if $(\cX_t, \cL_t)$ is $K$-polystable relative to $G$ then $(\cX_t,\cL_t)$ admits an extremal metric.
If we assume that  for each $t\neq 0$, $G$ is a maximal torus in $Aut(\cX_t)$, then this condition is also necessary.

\end{theorem}

\begin{proof}
The first part of the theorem follows \cite{sz}.
If $\cX_t$ admits an extremal metric, by Sz\'ekelyhidi and Stoppa \cite{ss09} it must be K-polystable with respect to a maximal torus of automorphisms.
\end{proof}

\subsection{Semiuniversal deformations and the slice}
\subsubsection{Equivariant deformations}
By Kodaira-Spencer theorem, if $H^2(X,\tsheaf_ X)=0$ there exists a
semiuniversal family of deformations $X\hookrightarrow \cX \rightarrow
B$ such that $B$ is an open neighborhood of the origin in 
$H^1(X,\tsheaf_ X)$, and the induced Kodaira-Spencer map 
$  T_0B\simeq
H^1(X,\tsheaf_{X}) \to H^1(X,\tsheaf_{X}) $ is the identity.

\begin{definition}
 Let $X\hookrightarrow \cX \rightarrow B$ be a deformation
 of $X$ and $H$ a compact connected Lie group in $\Aut(X)$ acting
 holomorphically on $\cX$ and satisfying the following properties
\begin{itemize}
\item $H$ acts in a fiber preserving manner on $\cX$, i.e. such that the action of $H$ descends to
$B$
 \item   $\cX_0$ is invariant under the
$H$-action, so that there is a morphism $H\to \Aut(X)$. 
\item the above morphism is the canonical inclusion $H\subset\Aut(X)$.
\end{itemize}
Such a deformation shall be called  a \textit{H-equivariant
  deformation} of $X$.
If a Lie subgroup $G$ in $H$ induces a trivial action on $B$, we say
that the $H$-equivariant deformation is $G$-invariant. If $H=G$ we
simply say that \emph{the deformation is $G$-invariant}.
\end{definition}

An immediate generalization of Kodaira-Spencer theory is given by the
following lemma:
\begin{lemma}
\label{lemma:ksH}
Let $X$ be a closed complex manifold satisfying  $H^2(X,\tsheaf_
X)=0$ and let $H$ be a compact connected Lie group in $\Aut(X)$.

Then there exists a semiuniversal family of complex deformations (in
the usual sense)
 $X\hookrightarrow \cX \rightarrow B$ which is $H$-equivariant.
Moreover, we may assume that $B$ is an open neighborhood of the origin
in $H^1(X,\tsheaf_
X)$ such that
\begin{enumerate}
\item  the induced Kodaira-Spencer map is the identity
\item the $H$-action on $B$ agrees with the canonical action of $H$
  on $H^1(X,\tsheaf_
X)$.
\end{enumerate}
In addition, the family of deformation is versal among
$H$-equivariant deformations. By this, we mean that any other
$H$-equivariant deformation $X\hookrightarrow \cX'\to B'$ is induced
by a $H$-equivariant holomorphic map $B'\to B$. 
\end{lemma}
\begin{proof}
  The proof is obtained using Kuranishi's approach \cite{Ku}, working
  with $H$-invariant metrics.
\end{proof}

\begin{definition}
\label{def:semiheq}
Given a closed complex manifold $X$ satisfying  $H^2(X,\tsheaf_
X)=0$ and $H$ be a compact connected Lie group in $\Aut(X)$, the
family satisfying the properties (1)-(2) of Lemma \ref{lemma:ksH} shall be
simply referred to as a \emph{$H$-equivariant semiuniversal family of
deformations} of $X$.  If $X$ is toric and $H$ is the real torus, the
family will be called instead a \emph{semiuniversal toric family of deformations}.
\end{definition}

\subsubsection{Properties of the  slice}
In favorable cases the tangent space $\widetilde{H}^1$ to the slice
agrees with the space of \emph{every} infinitesimal complex
deformations
 $H^1(X, TX)$.
This space is identified with 
$$
\lbrace \alpha\in \Om^{0,1}(T^{1,0}(X)), \delb\alpha=0, \delb^*\alpha = 0 \rbrace .
$$

\begin{lemma}
 \label{gcase}
If we assume $X=(M,J_0)$ simply connected, $H^2(X,\mathcal{O})=0$ and
$H^2(X,\tsheaf_ X)=0$, then $\widetilde{H}^1\simeq H^1(X,\tsheaf_ X)$. 
If $\om$ is a K\"ahler metric on $X$ and $H$ the group of Hamiltonian isometries of $(X,\om)$,
this isomorphism can be chosen $H$-equivariant.
\end{lemma}

\begin{proof}
We suppose that $X$ is simply connected. In that case, if $v$ satisfies
$$
\cL_v\om=0
$$
then there is $f\in C^{\infty}_0(M)$ such that ${v}={v}_f$.
Thus the equation $P^*(\alpha)=0$
is equivalent to $\delb^* \alpha =0$ for all $\alpha\in T_{J_0}\cJ$
and we see that $\widetilde{H}^1$ is the subspace of $H^1(X,\tsheaf_ X)$
consisting of elements $\alpha$ such that 
$$
\om(\alpha (v),u)+\om(v,\alpha(u))=0.
$$
The space $\widetilde{H}^1$ characterizes integrable infinitesimal
deformations of $J_0$ that are compatible with $\om$ up to symplectomorphisms.
Let $\xi\in H^1(X,\tsheaf_ X)$.
If $H^2(X,\tsheaf_ X)=0$ the deformation theory is unobstructed and
there exists a semi-universal family of deformations $X\hookrightarrow \cX \rightarrow B$ 
such that the image by the Kodaira-Spencer map of $1\in T_0 B$ is $\xi$. Then, by $H^2(X,\mathcal{O})=0$, we
know from Kodaira and Spencer theory that we can suppose the $\cX_t=(M,J_t)$ to be K\"ahler,
with the same cohomology class $[\om]$. Using Moser's trick, we get a new family 
$X\hookrightarrow \cX' \rightarrow B$ of deformations of $X$ such that $\cX'_t=(M,J_t',\om)$ is K\"ahler, that
is $J_t'$ is $\om$-compatible. The associated infinitesimal deformations corresponds to elements in $\widetilde{H}^1$.
By semi-universality of $\cX$, we have maps
$$
\begin{array}{ccc}
 \cX' & \longrightarrow & \cX \\
 \downarrow & & \downarrow \\
 B & \longrightarrow & B 
\end{array}
$$
and $\xi$ corresponds to an element of $\widetilde{H}^1$
via the tangent map $T_0B \rightarrow T_0B$, which proves that $\widetilde{H}^1\simeq H^1(X,\tsheaf_ X)$.
Note that these families of deformations can be chosen $H$-equivariant so that 
$T_0B \rightarrow T_0B$ is $H$-equivariant, with $H$ the group of Hamiltonian isometries of $X$.
This approach  has been used in \cite{bro}.
\end{proof}

\section{Deformations of extremal toric manifolds}
\label{toric}

It is possible to describe the deformation theory 
for a smooth compact extremal toric manifold in terms of combinatorial datum from its fan.
In particular, in the case of surfaces, we obtain a simple criterion giving a necessary and sufficient condition for
an extremal K\"ahler toric surface to admit deformations with extremal metrics. For basics on toric manifolds we refer to \cite{oda}.

\subsection{Action of the torus on the space of  infinitesimal deformations}

Let $X=TV(\Sigma)$ be a toric manifold of dimension $n$,
with $\Sigma$ a fan in a lattice $N$. We suppose $X$ compact and smooth.
In that case $X$ is simply connected and satisfies $H^2(X,\cO)=0$ \cite{oda}.
Endow $X$ with a toric metric $\om$ and let $H$ be the group of Hamiltonian isometries of $(X,\om)$.
Then if $H^2(X,\tsheaf_ X)=0$ we are interested in the action of $H$ on $H^1(X,\tsheaf_ X)$.
A result of Demazure \cite{dem} describes the group of automorphisms of toric varieties. In particular,
this group contains the torus $\TT^{\C}\simeq N\otimes_{\C} \C^*$ as a maximal torus.
We will restrict ourselves to the study of the action of $\TT\subset H$ on the space of infinitesimal deformations. 
In \cite{ilv}, Ilten and Vollmert gave a simple
description for generators of the vector space $H^1(X,\tsheaf_ X)$.

Let $N^*$ denote the dual of the lattice $N$. Then $H^1(X,\tsheaf_ X)$ is a $N^*$-graded algebra and admits a weight decomposition
$$
H^1(X,\tsheaf_ X)=\bigoplus_{R\in N^*} H^1(X,\tsheaf_ X)(R),
$$
where $H^1(X,\tsheaf_ X)(R)$ is the subspace of $H^1(X,\tsheaf_ X)$ on which the
$\TT^{\C}$-action is of weight $R$.
Let $\Sigma^{(1)}$ be the set of rays in $\Sigma$. To simplify notations, we will identify the rays of $\Sigma^{(1)}$ with primitive generators
of these rays.
Let $R\in N^*$ and $\rho\in \Sigma^{(1)}$ such that $\langle R, \rho \rangle =1$. 
Let $\Gamma_{\rho}(-R)$ be the graph embedded in $N_{\Q}$ with vertices consisting in primitive lattice generators
of rays $$\tau \in \Sigma^{(1)}\setminus \lbrace \rho \rbrace$$ such that $\langle \tau, R \rangle >0$. Two vertices are connected by an edge if they generate a cone in $\Sigma$. Now we let
$$
\Om(-R)=\lbrace \rho \in \Sigma^{(1)} / \langle \rho, R \rangle=1, \Gamma_{\rho}(-R) \neq \varnothing \rbrace.
$$
The relevant fact is that for each connected component $C$ of $\Gamma_{\rho}(-R)$,
Ilten and Vollmert constructed an element $\pi(C,\rho,R)$ of $H^1(X,\tsheaf_ X)(-R)$. Moreover, they
proved that these elements span $H^1(X,\tsheaf_ X)(-R)$ for $\rho\in \Omega(-R)$ and $C$ ranges over all connected
components of $\Gamma_{\rho}(-R)$. 
Then we can compute the action of the torus $T^{\C}$ on $H^1(X,\tsheaf_ X)$:

\begin{lemma}
\label{action}
Each space $H^1(X,\tsheaf_ X)(R)$ is fixed under the torus action on $H^1(X,\tsheaf_ X)$.
Moreover, the action of $(\lambda_1,...,\lambda_n)\in \TT^{\C}\simeq (\C^*)^n$ on $H^1(X,\tsheaf_ X)(R)$
is given by 
$$
\forall x \in H^1(X,\tsheaf_ X)(R),\; (\lambda_1,...,\lambda_n).x=\lambda_1^{\langle R, e_1\rangle}..\lambda_n^{\langle R, e_n\rangle}x
$$
with $(e_i)$ a $\Z$-basis for $N$.
\end{lemma}

\begin{proof}
From theorem 6.2. \cite{ilv}, for each $\rho \in \Omega(R)$ and each $C$ a connected component of $\Gamma_{\rho}(-R)$,
the element $\pi(C,\rho,R)$ is given as a cocycle by derivations defined on intersections of an open cover of $X$.
Each of these derivation is proportional to the derivation $\del (R,\rho)$ that takes
$$
\chi^v \mapsto \langle \rho, v \rangle \chi^{v+R}
$$
for $v\in N^*$ and where $\chi^{e_i^*}$ denotes the usual regular functions on the torus $Spec(\C[N^*])$.
Then the action of the torus on these derivations is
$$
\forall (\lambda_1,...,\lambda_n)\in \TT^{\C},
$$
$$
(\lambda_1,...,\lambda_n).\del(R,\rho)=\Pi_{i}\lambda_i^{\langle R, e_i\rangle} \del(R,\rho).
$$
To conclude, from theorem 6.5. of \cite{ilv}, the elements $\pi(C,\rho,R)$ span $H^1(X,\tsheaf_ X)(R)$.
\end{proof}

Now we can investigate polystable points under the action of the torus.

\subsection{Stability criteria}                                                                                    
\label{stable}
Let $N_{def}^*(\Sigma)$ be the subset of elements in $N^*$ satisfying
$$
\exists \rho\in \Sigma^{(1)}\;  /\; dim(H^0(\Gamma_{\rho}(R),\C)) \geq 2.
$$
Then from \cite{il} the weight decomposition under the torus action of $H^1(X,\tsheaf_ X)$ is:
$$
H^1(X,\tsheaf_ X)=\bigoplus_{R\in N^*_{def}(\Sigma)} H^1(X,\tsheaf_ X)(R).
$$
Recall that if $V=Spec(A)$ is an affine variety endowed with an algebraic action of a reductive group $G$, 
we form the GIT quotient $V/G=Spec(A^G)$ where $A^G$ is the ring of invariants.
Then the set of semi-stable points $V^{ss}$ is given by:
$$
V^{ss}=\lbrace x\in V / \exists P\in A^G/ P(x)\neq 0 \rbrace
$$
and the set of polystable points $V^p$ is the subset of points $x\in V^{ss}$ such that the orbit $G.x$
is closed in $V^{ss}$.

\begin{remark}
As our problem is settled in a linear context, the constant polynomials are invariant and each point is semi-stable.
However, if the only invariant polynomials that do not vanish on a point $x$ are the constants, then the ring of invariant polynomials makes
no difference between $x$ and $0$, thus this point is not polystable. We will first compute semi-stable points
that are detected by a non-constant polynomial and refer to such points as semi-stable points.
\end{remark}

In our situation, $G=\TT^{\C}$ and we consider $V=H^1(X,\tsheaf_ X)$.
Let denote $R_1,..,R_s$ the elements of $N^*_{def}(\Sigma)$ and $d_i$ the dimension of $H^1(X,\tsheaf_ X)(R_i)$.
Let $(E_{i,k})_{k=1..d_i}$ be a basis of $H^1(X,\tsheaf_ X)(R_i)$ such that $V=Spec(\C[X_{i,k}])$.

\begin{definition}
 We say that a family $\lbrace R_{1},..,R_{r})\in N^*$ is \textit{balanced}
if there exists 
$(a_i)\in \N^r, a_i\neq 0$ satisfying $ \sum_i a_iR_i=0$.
\end{definition}

For each balanced family $\textbf{R}=\lbrace R_k \rbrace\in N^*_{def}(\Sigma)$ , we set
$$
U_{\textbf{R}}=\lbrace x=x_1+..+x_s, x_i \in H^1(X,\tsheaf_ X)(R_i) / x_k\neq 0 \text{ for } R_k\in\textbf{R} \rbrace.
$$
Let $\nu(\Sigma)$ be the set of $I\subset \lbrace 1,..,s\rbrace$ such that $\lbrace R_i, i\in I \rbrace\in N^*_{def}(\Sigma)$ is a balanced family.
Then the set of semi-stable points $H^1(X,\tsheaf_ X)^{ss}$ is given by the following:
 
\begin{proposition}
 \label{ss}
There exist semi-stable points in $H^1(X,\tsheaf_ X)\setminus \lbrace 0 \rbrace$ under the action of $\TT^{\C}$ 
if and only if there is a balanced family in $N^*_{def}(\Sigma)$.
In that case,
$$
H^1(X,\tsheaf_ X)^{ss}\setminus \lbrace 0 \rbrace=\bigcup_{I\in \nu(\Sigma)} U_{\lbrace R_i, i\in I \rbrace}.
$$
\end{proposition}

\begin{proof}
 Let $V=H^1(X,\tsheaf_ X)$ and 
$$
\forall R\in N^*_{def}(\Sigma),\; W_R=H^1(X,\tsheaf_ X)(R).
$$
Let $P\in A=\C[X_{i,k}]$ and suppose that $P$ is not constant. 
Write
$$
P=\sum_J a_J X^J
$$
in a basis of $A$, with $X^J=X_1^{j_1}..X_r^{j_r}$.
Given the action of the torus on $V$ described in lemma~\ref{action}, we see that
$P\in A^G$ if and only if each component of $P$ is in $A^G$.
Thus we suppose that $P$ is written
$$
P=aX_{1,1}^{j_{1,1}}..X_{1,d_1}^{j_{1,d_r}}X_{2,1}^{j_{2,1}}..X_{2,d_2}^{j_{2,d_2}}...X_{s,d_s}^{j_{s,d_s}}.
$$
Then the action of $G$ on $P$ is:
$$
\forall \mathbf{\lambda}=(\lambda_1,...,\lambda_n)\in G,
$$
$$
\mathbf{\lambda}\cdot P= (\Pi_{i=1}^s
(\lambda_1^{\langle R_i, e_1 \rangle}..\lambda_n^{\langle R_i, e_n \rangle})^{\sum_{k=1}^{d_i} j_{i,k}} )P.
$$
Thus
$$
\mathbf{\lambda}\cdot P=
\lambda_1^{\langle \sum_i \sum_{k=1}^{d_i} j_{i,k} R_i, e_1 \rangle}..\lambda_n^{\langle \sum_i \sum_{k=1}^{d_i} j_{i,k}R_i, e_n \rangle} P,
$$
and 
$$
\forall \mathbf{\lambda}=(\lambda_1,...,\lambda_n)\in G, \; \mathbf{\lambda}\cdot P=P
$$
if and only if 
$$
\forall l\in \lbrace 1..n \rbrace,\; \langle \sum_i \sum_{k=1}^{d_i} j_{i,k} R_i, e_l \rangle=0 
$$
that is if and only if
$$
\sum_i \sum_{k=1}^{d_i} j_{i,k} R_i=0.
$$
We just proved that there exists semi-stable points in $V$ if and only if there exists a non trivial positive linear combination
of the elements $R_i$ that vanishes in $N^*$, which is a balanced family.
Moreover, we can describe the set $V^{ss}$ of semi-stable points in that case.
Let $(a_1,a_2,..,a_s)\in \N^s-\lbrace (0,0,..,0)\rbrace$ such that $\sum_i a_i R_i =0$. For each $i$, decompose
$a_j=\sum_k j_{j,k}$ into sum of integers (eventually zero). 
Set $P=\Pi X_{i,k}^{j_{i,k}}$. By construction, $P$ is $G$-invariant, and 
$$
\lbrace x\in V / P(x)\neq 0 \rbrace = \lbrace (x_{1,1},.., x_{s,d_s}) / x_{i,k}\neq 0 \text{ if } j_{i,k}\neq 0\rbrace.
$$
Then, the set of semi-stable points is the union of sets of this kind.

\end{proof}

We now describe the set of polystable points $H^1(X,\tsheaf_ X)^p$.
For each $R\in N^*- \lbrace 0 \rbrace$, define 
$$
\lbrace R<0 \rbrace=\lbrace x\in N / \langle R, x \rangle <0 \rbrace
$$
and 
$$
\lbrace R=0 \rbrace=\lbrace x\in N / \langle R, x \rangle =0 \rbrace.
$$
Let $\mu(\Sigma)$ be the set of all
$I\subset \lbrace 1,..,s \rbrace$, such that
$$
\exists k_1,..,k_r \in I \;/\; \lbrace R_{k_1},..,R_{k_r} \rbrace \text{ is a balanced family}
$$
and
$$
N = (\cup_{i\in I} \lbrace R_i<0 \rbrace) \bigcup (\cap_{i\in I}  \lbrace R_i=0 \rbrace ).
$$
For each family of indices $I\subset \{1,\cdots,s\}$, we consider the direct sum $\bigoplus
_{i\in I} H^1(X,\tsheaf_ X)(R_i)$ and we decompose each vector $x=\sum
x_i$  with $x_i\in H^1(X,\tsheaf_ X)(R_i)$.
Let $V_I$ be the finite union of
subvector spaces given by the equations $x_i=0$ for some $i\in I$.
Put
$$
S_I=\bigoplus _{i\in I} H^1(X,\tsheaf_ X)(R_i) \setminus V_I,
$$
Then the set of polystable points $H^1(X,\tsheaf_ X)^{p}$ satisfies the following:

\begin{proposition}
\label{p}          
 There exist polystable points in $H^1(X,\tsheaf_ X)\setminus \lbrace 0 \rbrace$ under the action of $\TT^{\C}$ 
if and only if there is a balanced family in $N^*_{def}(\Sigma)$.
In that case, 
$$
H^1(X,\tsheaf_ X)^{p}\setminus \lbrace 0 \rbrace=\bigcup_{I\in \mu(\Sigma)} S_I.
$$
\end{proposition}

\begin{proof}
We keep notations of the proof of proposition~\ref{ss}.
 The set of polystable points is the subset of semistable points $x$ such that the orbit $G\cdot x$ is closed in
$V^{ss}$.
Let $x\in V^{ss}$, $x=x_1+..+x_s, x_i\in W_{R_i} $.
Let $I_x=\lbrace i\vert x_i \neq 0 \rbrace$.
By proposition~\ref{ss}, there is $\lbrace i_1,..,i_r \rbrace \in \nu(\Sigma)$ such that $\lbrace i_1,..,i_r \rbrace \subset I_x$.
By the Hilbert-Mumford criterion, the orbit $G\cdot x$ is closed in $V^{ss}$
if and only if for each one-parameter subgroup $\C^*$ of $G$, 
the orbit $\C^* \cdot x$ is closed in $V^{ss}$.
One parameter subgroup of $G$ can be represented by
$$
\lambda_\textbf{p} \in \C^* \mapsto (\lambda^{p_1}, ... , \lambda^{p_n})\in \TT^{\C}
$$
for some $\textbf{p}=(p_1,..,p_n)\in \Z^n$.
For each $\textbf{p}\in \Z^n$, the action of the associated one-parameter
subgroup is
$$
\lambda_{\textbf{p}}\cdot x= \sum_j \lambda^{\langle R_j, p_1e_1 + ..+p_n e_n\rangle} x_j.
$$
To test closedness, it is enough to understand what happens when $\lambda$ tends to zero for each $\textbf{p}\in \Z^n$.
We can fix $(a_{i_k})\in N^{*r}$ such that $a_{i_1}R_{i_1}+..+a_{i_r}R_{i_r}=0$, thus 
$$
\langle a_{i_1}R_{i_1}+..+a_{i_r}R_{i_r}, p_1e_1 + ..+p_n e_n\rangle=0
$$
and $\exists (l,l')$ such that
$$
\langle R_{i_l}, p_1e_1 + ..+p_n e_n\rangle\langle R_{i_{l'}}, p_1e_1 + ..+p_n e_n\rangle <0
$$
unless
$$
\forall i_k \: \langle R_{i_k}, p_1e_1 + ..+p_n e_n\rangle=0 .
$$
In the first case, there is no limit in $V$ when $\lambda$ tends to zero.
In the second case, there is a limit $x_{\infty}$ in $V$ if and only if
$$
\forall j\in I_x,\; \langle R_j, p_1e_1 + ..+p_n e_n\rangle \geq 0
$$
and as 
$$
\forall i_k \; \langle R_{i_k}, p_1e_1 + ..+p_n e_n\rangle=0 
$$
this limit satisfies $x_{\infty i_k}\neq 0$ so $x_{\infty}\in V^{ss}$.
Thus we see that $x$ is polystable if and only if $\forall p\in N$, one of the following is satisfied
$$
\exists j \in I_x / \langle R_j, p\rangle < 0
$$
or
$$
\forall j \in I_x / \langle R_j, p\rangle = 0.
$$
That is exactly saying that 
$$
N = (\cup_{i\in I_x} \lbrace R_i<0 \rbrace) \bigcup (\cap_{i\in I_x}  \lbrace R_i=0 \rbrace )
$$
and
$$
I_x\in \mu(\Sigma)
$$
and we have $x\in S_{I_x}$.
To conclude the proof, note that if $\lbrace R_{i_1},..,R_{i_r})\in N^*$ is a balanced family, 
the point $ x_{i_1}+..+x_{i_r} \in W_{i_1} \oplus .. \oplus W_{i_r}$ with $ \;x_{i_k}\neq 0$ for all $k$
is polystable.
\end{proof}

In order to deal with complex deformations that preserve some torus action,
we need relative stability results. We will now describe polystability results
of subspaces of $H^1(X,\tsheaf_ X)$ that are fixed by a sub-torus action.
For each $p\in N$, $x\in\ V$ is fixed by the action of the corresponding
one-parameter subgroup if and only if 
$$
\forall j/ x_j \neq 0, \langle R_j, p \rangle = 0.
$$
Let's consider a splitting
$$
N=N_f\oplus N_a.
$$
It induces a decomposition of the torus
$$
\TT^{\C}=\TT_f \times \TT_a
$$
with
$$
\TT_a=N_a\otimes_{\C} \C^* \text{ and } \TT_f=N_f\otimes_{\C}\C^*.
$$
If $f_1, .. , f_d$ is a basis for $N_f$,
then the fixed set of $\TT_f$ is
$$
H^1(X,\tsheaf_ X)^{\TT_f}=\lbrace x \in H^1(X,\tsheaf_ X) \vert \forall j / x_j\neq 0, \forall l\in\lbrace 1 ,.., d\rbrace, \langle R_j, f_l \rangle = 0 \rbrace.
$$
and $\TT^a$ acts on $H^1(X,\tsheaf_ X)^{\TT_f}$.
Every $p\in N$ can be written $p=p_f+p_a \in N_f\oplus N_a$ and for every $x\in H^1(X,\tsheaf_ X)^{\TT_f}$,
$$
\lambda_p\cdot x= \lambda_{p_a+p_f}\cdot x= \lambda_{p_a}\cdot \lambda_{p_f} \cdot x= \lambda_{p_a} \cdot x.
$$
Thus the stability with respect to every one-parameter subgroup of $\TT^{\C}$ is equivalent to the stability
with respect to every one-parameter subgroup of $\TT_a$ on $H^1(X,\tsheaf_ X)^{\TT_f}$.
Let
$$
 N^*_{\TT_f}(\Sigma)=\lbrace R\in N^*_{def}(\Sigma) / \forall l\in\lbrace 1 ,.., d\rbrace, \langle R, f_l \rangle = 0 \rbrace
$$
and
$$
\mu_{\TT_f}(\Sigma)=\lbrace I \in \mu(\Sigma) / \forall i\in I, R_i \in N^*_{\TT_f}(\Sigma)\rbrace.
$$
Then, the results of proposition~\ref{ss} and proposition~\ref{p} imply

\begin{proposition}
\label{relativep}
 There exist polystable points in $H^1(X,\tsheaf_ X)^{\TT_f}\setminus \lbrace 0 \rbrace$ under the action of $\TT_a$ 
if and only if there is a balanced family in $N^*_{\TT_f}(\Sigma)$.
In that case, the set of polystable points is
$$
H^1(X,\tsheaf_ X)^{\TT_f p}= \lbrace 0 \rbrace \cup \bigcup_{I\in\mu_{\TT_f}(\Sigma) } S_I.
$$
\end{proposition}

\begin{remark}
 The description of stable points in $\P(V)$ under a torus action given by a representation on a vector space $V$ 
is given by Sz\'ekelyhidi in terms of a weight polytope, \cite{szt}. Our results are closely related to this description.
\end{remark}

\subsection{Existence of toric extremal deformations}
Using the general setup of section~\ref{deformation} and the stability criteria of section~\ref{stable},
we are now able to prove our main results on deformations of extremal toric manifolds.

\begin{theorem}
\label{equivextremal}
Let $X=TV(\Sigma)$ be a smooth compact toric manifold endowed with an extremal toric
K\"ahler structure $(J,\om)$.
Let $H$ be the group of Hamiltonian isometries of $(J,\om)$ and assume $H^{\C}=\TT^{\C}$.

Suppose that $H^2(X,\tsheaf_ X)=0$ and consider the semiuniversal toric family of deformations
$X\hookrightarrow \cX\to
B$ of
$X\simeq \cX_0$ with $B$ identified to a  ball centered at the
origin in $H^1(X,\tsheaf_
X)$.

Suppose that the extremal vector field is contained in the Lie algebra of a torus $\TT_f \subset\TT^{\C}$ .
In that case, 
for each $t$ small enough in
$$
\lbrace 0 \rbrace \cup \bigcup_{I\in \mu_{\TT_f}(\Sigma)} S_I
$$
$\cX_t$ admits an extremal metric.

If $[\om]$ represents a polarization $L$ of $X$ we can suppose $\cX$ to
be polarized by $\cL$. Then if $\TT_f$
is a maximal torus of automorphisms of $\cX_t$, $t$ belongs to $\lbrace 0 \rbrace \cup \bigcup_{I\in \mu_{\TT_f}(\Sigma)} S_I$
if and only if $\cX_t$ admits an extremal metric in the class $c_1(\cL_t)$
\end{theorem}

\begin{proof}
Recall that $X$ is simply connected and that $H^2(X,\mathcal{O})=0$. Together with the hypothesis $H^2(X,\tsheaf_ X)=0$,
by the lemma~\ref{gcase}, we know that the equivariant slice constructed 
in section~\ref{reduction} corresponds to a map from a neighborhood of zero in $H^1(X,\tsheaf_ X)^{\TT_f}$
to the space of $\om$-compatible and $\TT_f$-invariant integrable complex structures on the underlying differentiable manifold.
By proposition~\ref{linearstability}, there is a neighborhood $U$ of zero in $ H^1(X,\tsheaf_ X)^{\TT_f}$
such that every polystable point in $U$ under the action of $\TT^{\C}/\TT_f$ gives rise to an extremal metric on
the corresponding complex manifold.
Then the description of polystable points in proposition~\ref{p} ends the proof of the first part of the theorem.
The existence of a balanced family in $N^*_{\TT_f}(\Sigma)$
is equivalent to the existence of polystable points in $H^1(X,\tsheaf_ X)^{\TT_f}$ and the last part of the theorem follows from the discussion of
the section~\ref{sec:polarized}
\end{proof}

\begin{remark}
The existence of projective deformation endowed with extremal metric is thus equivalent to the existence of a balanced family
in the space $H^1(X,\tsheaf_X)^\TT_f$. This can be interpreted as a rigidity result for polarized extremal metrics.
Note that the stability
condition for the existence of an extremal projective deformation
does not depend on the K\"ahler class.
\end{remark}

\begin{remark}
  If $H^2(X,\Theta_X)\neq 0$, the deformation theory could be
  obstructed.
Following Kuranishi \cite{Ku}, the set of integrable complex structures
in the slice corresponds to the vanishing locus in $B\subset
H^1(X,\Theta_X)$ of a holomorphic
function $B\to H^2(X,\Theta_X)$. The next step would be to understand
whether polystable orbits intersect (and therefore are contained in)
the vanishing locus of the obstruction map.

 This general situation seems rather pathological and very
 interesting. However little hope is left for constructing examples since experts in the field of
 $T$-varieties  expect the obstruction map to be identically zero in
 the toric case \cite{iltp}.
\end{remark}

\subsection{Deformation of extremal toric surfaces}

The case of surfaces deserves special attention as it admits an even simpler formulation.
First of all, from corollary 1.5. \cite{il} of Ilten, $H^2(X,\tsheaf_ X)=0$ and the deformation theory is unobstructed.

Moreover, the space $H^1(X,\tsheaf_ X)$ admits a simpler description.
Let's number the rays of $\Sigma^{(1)}$ by $ \rho_1,..,\rho_l$ and $\rho_{l+1}=\rho_1$.
From corollary 1.5. \cite{il}, we have
$$
N^*_{def}(\Sigma)=\lbrace R\in N^*/ \exists \rho_i\in\Sigma^{(1)}/ \langle \rho_i, R\rangle =-1 \text{ and } \langle \rho_{i\pm1}, R\rangle <0 \rbrace
$$
so that it is easy to understand polystable points. We will proceed to explicit computations in the following section.

It is also easy to understand the restriction $H^{\C}=\TT^{\C}$ needed in the deformation of
CSC metrics in the case of surfaces.
We suppose that the toric surface $X=TV(\Sigma)$ is endowed with an extremal metric.
By Calabi's theorem, the group of Hamiltonian isometries is a maximal compact subgroup of $\Aut(X)$.
Up to conjugation, we can suppose that $\TT^{\C}\subset H^{\C}$ and we want to understand when the equality holds.
Every smooth compact surface is a successive equivariant blow-up of $\P^2$ or $\F_a$, the $a^{\text{th}}$ Hirzebruch surface.
As $\P^2$ is rigid and a one point blow-up of $\P^2$ is isomorphic to $\F_1$, we restrict our attention
to the $\F_a$s, $a\geq 0$.

A result of Demazure \cite{dem} describes the automorphism group of a compact non-singular toric manifold.
In particular, the Lie algebra of $\Aut(X)$ can be decomposed in the following manner
$$
Lie(\Aut(X))=Lie(\TT^{\C})\oplus \mathfrak{V}
$$
with $\mathfrak{V}$ a vector space generated by vector fields in one to one correspondence with the \textit{root system} of the fan:
$$
R(N,\Sigma)=\lbrace \alpha \in N^* / \exists \rho \in \Sigma^{(1)} / \langle \alpha, \rho \rangle =1 \text{ and } \langle \alpha, \rho' \rangle \leq 0 
\text{ for } \rho'\in\Sigma^{(1)}, \rho'\neq 0 \rbrace .
$$
As $\F_0=\C\P^1\times\C\P^1$, it is rigid.
For $a>0$, $\F_a$ can be endowed with one of Calabi's extremal metric in each K\"ahler class.
In that case, the extremal vector field is by construction vertical in the fibration 
$$
\F_a \rightarrow \C\P^1.
$$
Let $\Sigma_a$ be the complete fan associated to $\F_a$ in the lattice $N=\Z e_1\oplus \Z e_2$ with
$$
\Sigma^{(1)}_a=\lbrace e_1, e_2, -e_2, -e_2-ae_1 \rbrace 
$$
where again we identify rays with their primitive generators.
In that case the vertical action is generated by $\C^*\otimes_{\Z} \Z e_2$.
Then we compute 
$$
N_{def}^*(\Sigma_a)=\lbrace x e_1^*+e_2^*, 1-a \leq x \leq -1 \rbrace
$$
And 
$$
N^*_{\TT_{e_2}}(\Sigma_a)= \varnothing.
$$
Thus there is no polarized family of deformation of Calabi's extremal metric.
In the sequel, we will consider toric surfaces that are obtained from $\F_a$ by at least one blow-up.
Recall that if $(\rho_i)$ denotes the rays of the fan of a toric surface $X$,
for each $\sigma=\R^{+*}\rho_j\oplus\R^{+*}\rho_k$ we can define a fixed-point set
of the torus action $V_{\rho_j,\rho_k}=0$ in $Spec(\C[X \cap \lbrace x/x_{\vert \sigma}\geq 0\rbrace])\subset X$.
Then the one point equivariant blow-up of $X$ at the point $V_{\rho_j,\rho_k}$ is described by the coarsest fan containing the $(\rho_i)$ and $\rho_j+\rho_k$. We will say that $\rho_j+\rho_k$
is a ray obtained from a blow-up of $X$.
Then we have

\begin{proposition}
\label{blup}
 Let $X=TV(\Sigma)$ be a toric surface obtained from $\F_a$ by $k$ blow-ups, $k\geq 2$.
Let $(\rho_k)$ denote the generators of the rays obtained from the blow-ups of $\F_a$.
If there exists $(\rho_1,\rho_2)$ such that
$$
\langle e_1^*,\rho_1 \rangle >0 \text{ and } \langle -e_1^*, \rho_2 \rangle >0 \text{ if } a\geq 1
$$
or
$$
\rho_1=-\rho_2 \text{ if } a=0
$$
then the complexification of the maximal subgroup of $X$ is the torus $\TT^{\C}$.
\end{proposition}

Then, if we start from $X$ a well chosen two points blow-up of some $\F_a$, any blow-up of $X$ will satisfy
the hypothesis required in our deformation results on extremal metrics.

\begin{proof}
 First, we can relate the root system of $X$ with the root system of $\F_a$.
From proposition 3.15. \cite{oda},
$$
R(N,\Sigma)=\lbrace \alpha \in R(N,\Sigma_a) / \forall k, \langle \alpha, \rho_k \rangle \leq 0 \rbrace.
$$
For $a=0$, 
$$
R(N,\Sigma_0)=\lbrace e_1^*,-e_1^*,e_2^*,-e_2^* \rbrace
$$
and the hypothesis in the theorem implies $R(N,\Sigma)=\varnothing$, which implies the result by Demazure's structure theorem.
In the $a\geq 1$ case, from \cite[chapter 9]{gbook}, the automorphism group of $\F_a$ is
$$
Aut(\F_a)\simeq GL_2(\C)/\mu_a \ltimes H^0(\C\P^1, \mathcal{O}(a))
$$
where $\mu_a$ denotes the group of $a^{\text{th}}$ roots of unity. Its maximal
compact subgroup $K_a$ is conjugated to
$$
K_a=U(2)/\mu_a
$$
Then the complexification of a maximal compact subgroup of automorphism is, up to conjugation,
$$
K_a^{\C}=GL_2(\C)/\mu_a.
$$
This is a four dimensional group that contains the torus as a subgroup.
Its Lie algebra contains the Lie algebra of the torus and two other generators corresponding
to two elements of the root system. These elements can be identified as those who leaves globally invariant the zero and infinity sections of 
the ruling of $\F_a$
over $\C\P^1$.
Lets denote $Z=\chi^{e_1^*}$ and $Y=\chi^{e_2^*}$ so that $\F_a$ is obtained
by gluing the four affine charts
$$
X_1=Spec(\C[Z,Y]), X_2=Spec(\C[Z,Y^{-1}]),
$$
$$
X_3=Spec(\C[Z^{-1},Z^{-a}Y]) \text{ and } X_4=Spec(\C[Z^{-1},Z^aY^{-1}])
$$
corresponding to the four cones
$$
\R^{+*}e_1\oplus\R^{+*}e_2, \R^{+*}e_1\oplus\R^{+*}-e_2,
$$
$$
\R^{+*}-e_1-ae_2\oplus\R^{+*}e_2 \text{ and } \R^{+*}-e_1-ae_2\oplus\R^{+*}-e_2.
$$
Then the zero section is given by
$$
\lbrace Y^{-1}=0 \rbrace \cup \lbrace Z^aY^{-1}=0 \rbrace \subset X_2\cup X_4
$$
and the infinity section by
$$
\lbrace Y=0 \rbrace \cup \lbrace Z^{-a}Y=0 \rbrace \subset X_1\cup X_3.
$$
Then, 
$$
R(N,\Sigma_a)=\lbrace e_1^*,-e_1^*, ke_1^*+e_2^* \text{ for } -a\leq k \leq 0 \rbrace
$$
and the $\C$-action induced by $e_1^*$ and $-e_1^*$ on the coordinate functions is computed given Demazure's formula \cite{dem}
$$
\forall \lambda\in \C, e_1^*(\lambda)\cdot Z^pY^q= Z^p(1+\lambda Z)^{-p}Y^q(1+\lambda Z)^{-aq}
$$
and 
$$
\forall \lambda\in \C, e_1^*(\lambda)\cdot Z^pY^q= (Z+\lambda)^pY^q.
$$
In particular, these actions preserve globally the zero and infinity sections, thus $K_a^{\C}$
is generated by the torus and the two groups corresponding to these actions.
The description of $R(N,\Sigma)$ in terms of $R(N,\Sigma_a)$ shows that the hypothesis of the proposition implies
that $e_1^*$ and $-e_1^*$ does not belong to $R(N,\Sigma)$ anymore. Then it only remains the torus in the complexification
of the maximal compact subgroup of $\Aut(X)$.
\end{proof}

\section{Applications}
\label{examples}
We apply the previous results to extremal toric surfaces.

\subsection{Deformations of CSC metrics}
\label{exampleCSC}

We begin this section by the construction of a special family of CSC toric surfaces.
Consider $\C\P^1\times\C\P^1$ endowed with a CSC K\"ahler metric.
Then $\Z_q$ acts by isometries on $\C\P^1\times \C\P^1$:
$$
\forall \xi\in\mu_q, \xi\cdot([x_1,y_1],[x_2,y_2])=([\xi x_1,y_1],[\xi x_2,y_2]).
$$
and the inversion 
$$
\begin{array}{cccc}
 I : & \C\P^1\times \C\P^1 & \rightarrow & \C\P^1\times \C\P^1 \\
&([x_1,y_1],[x_2,y_2]) & \mapsto & ([y_1,x_1],[y_2,x_2]) \\
\end{array}
$$
descends to an isometry on the quotient
$$
\C\P^1\times \C\P^1/\Z_q.
$$
Then, by a result of Rollin and Singer \cite{rs}, the toric resolution $\Xhat$ of $\C\P^1\times \C\P^1/\Z_q$
admits a CSC metric $\om$. This result is based on a gluing construction and working modulo $I$ ensures that all obstruction vanishes.
We want to apply our deformation theory to $\Xhat$ and we need a fan description of this toric manifold.
We compute it in the case $q=3$ but the method and the results extend for all $q\geq 2$.
$\C\P^1\times\C\P^1$ is described by the fan $\Sigma_0$ with:
$$
\Sigma_0^{(1)}=\lbrace e_1,-e_1,e_2,-e_2 \rbrace.
$$
It is represented by Figure 1.
\begin{figure}[htbp]
\psset{unit=0.85cm}
\begin{pspicture}(-4.5,-4.5)(4.5,0)
$$
\xymatrix @M=0mm{
\bullet & \bullet & \bullet & \bullet & \bullet \\
\bullet & \bullet & \bullet & \bullet & \bullet \\
 \bullet & \bullet & \bullet\ar[u]^{e_2}\ar[r]^{e_1}\ar[l]\ar[d]& \bullet & \bullet \\
\bullet & \bullet & \bullet & \bullet & \bullet \\
\bullet & \bullet & \bullet & \bullet & \bullet 
}
$$
\end{pspicture}
\caption{$\C\P^1\times\C\P^1$}
\end{figure}
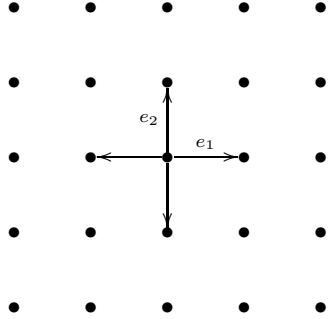

An affine open cover is given by
$$
X_1=Spec(\C[Z,Y]), X_2=Spec(\C[Z,Y^{-1}]),
$$
$$
X_3=Spec(\C[Z^{-1},Y]) \text{ and } X_4=Spec(\C[Z^{-1},Y^{-1}])
$$
corresponding to the four cones
$$
\R^{+*}e_1\oplus\R^{+*}e_2, \R^{+*}e_1\oplus\R^{+*}-e_2,
$$
$$
\R^{+*}-e_1\oplus\R^{+*}e_2 \text{ and } \R^{+*}-e_1\oplus\R^{+*}-e_2.
$$
Then the action of $\Z_3$ reads
$$
(Z,Y) \mapsto (\xi Z,\xi Y)
$$
so that the fan $\Sigma_s$ of $\C\P^1\times \C\P^1/\Z_3$
is given by the coarsest fan with
$$
\Sigma^{(1)}_s=\lbrace e_1, e_1+ 3 e_2, -e_1, -e_1-3e_2 \rbrace,
$$
(Figure 2).

\begin{figure}[htbp]
\psset{unit=0.85cm}
\begin{pspicture}(-4.5,-7)(4.5,0.5)
$$
\xymatrix @M=0mm{
\bullet & \bullet & \bullet & \bullet & \bullet \\
\bullet & \bullet & \bullet & \bullet & \bullet \\
\bullet & \bullet & \bullet & \bullet & \bullet \\
 \bullet & \bullet & \bullet\ar[r]\ar[ruuu]\ar[l]\ar[lddd]& \bullet & \bullet \\
\bullet & \bullet & \bullet & \bullet & \bullet \\
\bullet & \bullet & \bullet & \bullet & \bullet \\
\bullet & \bullet & \bullet & \bullet & \bullet 
}
$$
\end{pspicture}
\caption{$\C\P^1\times \C\P^1/\Z_3$}
\end{figure}
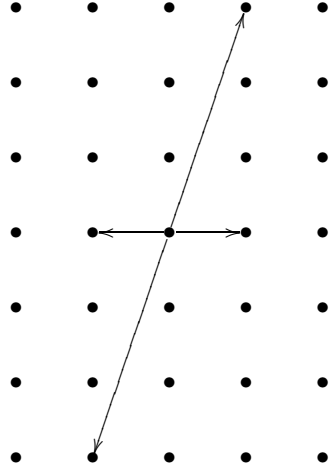

Indeed, an open cover for the toric manifold associated to $\Sigma_s$ is
$$
X_{s1}=Spec(\C[U,W,U^3W^{-1}]), X_{s2}=Spec(\C[W^{-1},UW^{-1},U^2W^{-1},U^3W^{-1}]),
$$
$$
X_{s3}=Spec(\C[W,U^{-1}W,U ^{-2}W,U^{-3}W]) \text{ and } X_{s4}=Spec(\C[U^{-1},W^{-1},U^{-3}W])
$$
and the change of variables 
$$
U=YZ^{-1} \text{ and } W^{-1}=Z^3
$$
shows that if $X_i=Spec(A_i)$, then $X_{si}=Spec(A_i^{\Z_3})$.
Then, the toric minimal resolution $\Xhat$ is described by the fan $\Sigma$ with
$$
\Sigma^{(1)}=\lbrace e_1, e_1+e_2, e_1+2e_2, e_1+3e_2,e_2, -e_1, -e_1-e_2,- e_1-2e_2, -e_1-3e_2,-e_2\rbrace,
$$
represented by figure 3.

\begin{figure}[htbp]
\psset{unit=0.85cm}
\begin{pspicture}(-4.5,-7)(4.5,0)
$$
\xymatrix @M=0mm{
\bullet & \bullet & \bullet & \bullet & \bullet \\
\bullet & \bullet & \bullet & \bullet & \bullet \\
\bullet & \bullet & \bullet & \bullet & \bullet \\
 \bullet & \bullet & \bullet\ar[r]\ar[ru]\ar[ruu]\ar[ruuu]\ar[u]\ar[l]\ar[ld]\ar[ldd]\ar[lddd]\ar[d]& \bullet & \bullet \\
\bullet & \bullet & \bullet & \bullet & \bullet \\
\bullet & \bullet & \bullet & \bullet & \bullet \\
\bullet & \bullet & \bullet & \bullet & \bullet 
}
$$
\end{pspicture}
\caption{Resolution of $\C\P^1\times \C\P^1/\Z_3$}
\end{figure}
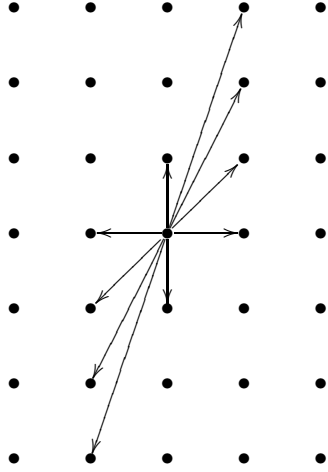
We recognize a twice three-times iterated blow-up of $\C\P^1\times \C\P^1$.
From proposition~\ref{blup}, the complexification of the maximal compact subgroup of $\Aut(\Xhat)$
is $\TT^{\C}$.
Moreover, we compute
$$
H^1(\Xhat,\tsheaf_ {\Xhat})= H^1(e_1^*)\oplus H^1(-e_1^*)\oplus H^1(2e_1^*-e_2^*)\oplus H^1(-2e_1^*+e_2^*)
$$
$$
\oplus H^1(e_2^*-e_1^*)
\oplus H^1(-e_2^*+e_1^*).
$$
As $(e_1^*,-e_1^*)$ forms a balanced pair, from theorem~\ref{equivextremal}, $\Xhat$ endowed with the CSC metric $\om$
admits projective CSC deformations.
Moreover, we see that $\Xhat$ admits $\C^*$-equivariant projective CSC deformations.
For example a point $x_1+y_1$ with $x_1\in H^1(e_1^*) \setminus \lbrace 0 \rbrace$, $y_1\in H^1(-e_1^*)\setminus \lbrace 0 \rbrace$,
and $\vert x_1+y_1\vert$ small enough generates deformations endowed with the $\C^*$-action generated by $e_2$.

\subsection{Description of the projective CSC deformations}
\label{sec:desccompat}
We can understand more precisely the deformations using the theory of $\TT$-invariant divisors developed in \cite{ps}.
Let's consider the simplest example from last section, which is also the example considered in the introduction, section \ref{sec:exintro}.
Let $\Xhat$ be the resolution of $\C\P^1\times \C\P^1 / \Z_2$. Following section~\ref{exampleCSC}
we can endow $\Xhat$ with a CSC metric.
The fan description of $\Xhat$ is on Figure 4.

\begin{figure}[htbp]
\psset{unit=0.85cm}
\begin{pspicture}(-4.5,-2.8)(4.5,0)
$$
\xymatrix @M=0mm{
\bullet & \bullet & \bullet & \bullet & \bullet \\
 \bullet & \bullet & \bullet\ar[r]\ar[ru]\ar[u]\ar[lu]\ar[l]\ar[ld]\ar[d]\ar[rd]& \bullet & \bullet \\
\bullet & \bullet & \bullet & \bullet & \bullet 
}
$$
\end{pspicture}
\caption{Resolution of $\C\P^1\times \C\P^1/\Z_2$}
\end{figure}
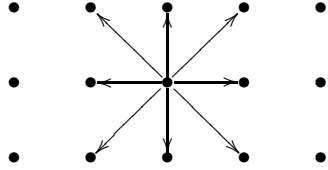

This toric variety is the blow-up of $\C\P^1\times\C\P^1$ at the four fixed points under the standard torus action.
Then we compute
$$
H^1(\Xhat,\tsheaf_ {\Xhat})=H^1(e_1^*)\oplus H^1(-e_1^*) \oplus H^1(e_2^*) \oplus H^1(-e_2^*).
$$
and $\Xhat$ admits projective CSC deformations.

Let's denote the weights of the torus action by $R_i=e_i^*$ and $R_{-i}=e_{-i}^*$ for $i\in\lbrace 1,2\rbrace$. Then $\mu(\Sigma)=\lbrace \lbrace 1,-1\rbrace ,  \lbrace 2,-2\rbrace, \lbrace 1,-1,2,-2\rbrace \rbrace$.  Then
$$
S_{\lbrace -1,1\rbrace}=\left ( H^1(e_1^*)\oplus H^1(-e_1^*)\right ) \setminus
\left \lbrace (x_1,x_2)\in H^1(e_1^*)\oplus H^1(-e_1^*) / x_1x_2 =0
\right \rbrace,
$$
$$
S_{\lbrace -2,2\rbrace}=\left (H^1(e_2^*)\oplus H^1(-e_2^*)\right )
\setminus\left \lbrace (x_3,x_4)\in H^1(e_2^*)\oplus H^1(-e_2^*) /
  x_3x_4 =0\right \rbrace
$$
and
$$
S_{\lbrace -1,1,-2, 2\rbrace}=\left (
 H^1(e_1^*)\oplus H^1(-e_1^*)\oplus H^1(e_2^*)\oplus H^1(-e_2^*)
\right )
\setminus \left \lbrace x_1x_2x_3x_4=0\right \rbrace.
$$
The set of non-zero polystable points in  $H^1(\Xhat,\tsheaf_
{\Xhat})$ under the torus action is the union of these sets and we recover the
description of the typical example given in \S \ref{sec:exintro} with
the notations
\begin{equation*}
  \label{eq:intronot}
S_{\lbrace -1,1\rbrace}=U_2', \quad S_{\lbrace -2,2\rbrace}=U_2''
\quad \mbox{ and } \quad  S_{\lbrace -1,1,-2, 2\rbrace}=U_4.  
\end{equation*}

Let's describe the deformations corresponding to $R=e_2^*$ and $\rho=-e_2$ in $H^1(e_2^*)$.
This deformation preserves the $\C^*$ action induced by $e_1$.
We start by down-grading the torus action to a circle action generated by $e_1$ in order to see $\Xhat$ as a $\C^*$-variety.
This description is given by a \textit{divisorial fan} (see for example \cite{ps}), represented by Figure 5.

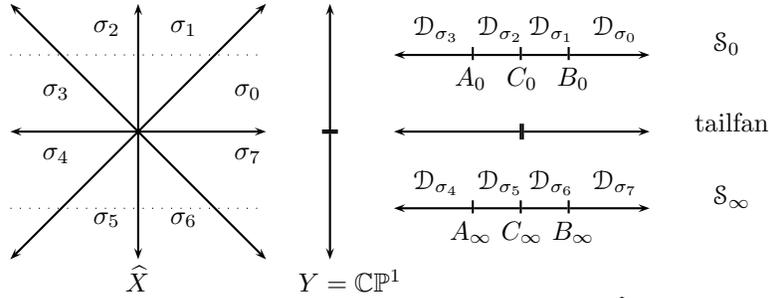
\begin{figure}[htbp]
\psset{unit=0.85cm}
\begin{pspicture}(0,-5)(12,0)
\psframe[linecolor=white](0.5,-4.5)(3.5,-1.5)

\psline{<->}(0,-1)(4,-5)
\psline{<->}(0,-5)(4,-1)
\psline{<->}(0,-3)(4,-3)
\psline{<->}(2,-5)(2,-1)
\psline[linewidth=0.5pt, linestyle=dotted]{-}(0,-1.8)(4,-1.8)
\psline[linewidth=0.5pt, linestyle=dotted]{-}(0,-4.2)(4,-4.2)

\rput[bl]{0}(3.5,-2.5){$\sigma_0$}
\rput[bl]{0}(2.5,-1.5){$\sigma_1$}
\rput[bl]{0}(1.3,-1.5){$\sigma_2$}
\rput[bl]{0}(0.5,-2.5){$\sigma_3$}
\rput[bl]{0}(0.5,-3.5){$\sigma_4$}
\rput[bl]{0}(1.3,-4.5){$\sigma_5$}
\rput[bl]{0}(2.5,-4.5){$\sigma_6$}
\rput[bl]{0}(3.5,-3.5){$\sigma_7$}
\rput[bl]{0}(1.8,-5.5){$\Xhat$}

\psline{<-|}(6,-3)(8,-3)
\psline{|->}(8,-3)(10,-3)

\rput[bl]{0}(10.7,-3){\textnormal{tailfan}}

\psline{<-|}(6,-1.8)(7.25,-1.8)
\psline{-|}(7.25,-1.8)(8,-1.8)
\psline{-|}(8,-1.8)(8.75,-1.8)
\psline{->}(8.75,-1.8)(10,-1.8)
\uput*[270](7.2,-1.8){$A_0$}
\uput*[270](8.8,-1.8){$B_0$}
\uput*[270](8,-1.8){$C_0$}
\rput[bl]{0}(11,-1.8){$\mathcal{S}_0$}
\rput[bl]{0}(9.1,-1.6){$\mathcal{D}_{\sigma_0}$}
\rput[bl]{0}(8.1,-1.6){$\mathcal{D}_{\sigma_1}$}
\rput[bl]{0}(7.3,-1.6){$\mathcal{D}_{\sigma_2}$}
\rput[bl]{0}(6.3,-1.6){$\mathcal{D}_{\sigma_3}$}
 
\psline{<-|}(6,-4.2)(7.25,-4.2)
\psline{-|}(7.25,-4.2)(8,-4.2)
\psline{-|}(8,-4.2)(8.75,-4.2)
\psline{->}(8.75,-4.2)(10,-4.2)
\uput*[270](7.2,-4.2){$A_{\infty}$}
\uput*[270](8.8,-4.2){$B_{\infty}$}
\uput*[270](8,-4.2){$C_{\infty}$}
\rput[bl]{0}(11,-4.2){$\mathcal{S}_{\infty}$}
\rput[bl]{0}(9.1,-4){$\mathcal{D}_{\sigma_7}$}
\rput[bl]{0}(8.1,-4){$\mathcal{D}_{\sigma_6}$}
\rput[bl]{0}(7.3,-4){$\mathcal{D}_{\sigma_5}$}
\rput[bl]{0}(6.3,-4){$\mathcal{D}_{\sigma_4}$}

\psline{|->}(5,-3)(5,-1)
\psline{|->}(5,-3)(5,-5)
\rput[bl]{0}(4.5,-5.5){$Y=\C\P^1$}
\end{pspicture}
\caption{Divisorial fan associated to $\Xhat$.}
\end{figure}
In the language of T-varieties, $\Xhat$ is the T-variety associated to the divisorial fan $\mathcal{S}$ on $Y=\C\P^1$,
with 
$$
\mathcal{S}=\lbrace  \mathcal{D}_{\sigma_0}\otimes 0 +  \mathcal{D}_{\sigma_7} \otimes \infty,\;
 \mathcal{D}_{\sigma_1}\otimes 0 +  \mathcal{D}_{\sigma_6} \otimes \infty,\;
$$
$$
\mathcal{D}_{\sigma_2}\otimes 0 +  \mathcal{D}_{\sigma_5} \otimes \infty,\;
\mathcal{D}_{\sigma_3}\otimes 0 +  \mathcal{D}_{\sigma_4} \otimes \infty \rbrace.
$$
Then the deformation associated to $e_2^*$ is a T-deformation described by the slice decomposition of figure 6 (see \cite{ilv}).
 \begin{figure}[htbp]
\psset{unit=0.85cm}
\begin{pspicture}(0,-6)(12,0)
\psframe[linecolor=white](0.5,-4.5)(3.5,-1.5)

\psline{<-|}(0,-1.6)(1.25,-1.6)
\psline{-|}(1.25,-1.6)(2,-1.6)
\psline{-|}(2,-1.6)(2.75,-1.6)
\psline{->}(2.75,-1.6)(4,-1.6)
\uput*[270](1.2,-1.6){$A_0$}
\uput*[270](2.8,-1.6){$B_0$}
\uput*[270](2,-1.6){$C_0$}
\rput[bl]{0}(4.4,-1.6){$\mathcal{S}_0$}
\rput[bl]{0}(3.1,-1.4){$\mathcal{D}_{\sigma_0}$}
\rput[bl]{0}(2.1,-1.4){$\mathcal{D}_{\sigma_1}$}
\rput[bl]{0}(1.3,-1.4){$\mathcal{D}_{\sigma_2}$}
\rput[bl]{0}(0.3,-1.4){$\mathcal{D}_{\sigma_3}$}
 
\psline{<-|}(0,-4.4)(1.25,-4.4)
\psline{-|}(1.25,-4.4)(2,-4.4)
\psline{-|}(2,-4.4)(2.75,-4.4)
\psline{->}(2.75,-4.4)(4,-4.4)
\uput*[270](1.2,-4.4){$A_{\infty}$}
\uput*[270](2.8,-4.4){$B_{\infty}$}
\uput*[270](2,-4.4){$C_{\infty}$}
\rput[bl]{0}(4.4,-4.4){$\mathcal{S}_{\infty}$}
\rput[bl]{0}(3.1,-4.2){$\mathcal{D}_{\sigma_7}$}
\rput[bl]{0}(2.1,-4.2){$\mathcal{D}_{\sigma_6}$}
\rput[bl]{0}(1.3,-4.2){$\mathcal{D}_{\sigma_5}$}
\rput[bl]{0}(0.3,-4.2){$\mathcal{D}_{\sigma_4}$}

\rput[bl]{0}(1.8,-5.7){$\Xhat$}

\psline{<-|}(6,-1.4)(7.25,-1.4)
\psline{-|}(7.25,-1.4)(8,-1.4)

\psline{->}(8,-1.4)(10,-1.4)
\uput*[270](7.2,-1.4){$A_0$}
\uput*[270](8,-1.4){$C_0$}

\rput[bl]{0}(11,-1.4){$\mathcal{S}_0$}
\rput[bl]{0}(9.1,-1.2){$\mathcal{D}^0_{\sigma_0}$}

\rput[bl]{0}(7.3,-1.2){$\mathcal{D}^0_{\sigma_2}$}
\rput[bl]{0}(6.3,-1.2){$\mathcal{D}^0_{\sigma_3}$}

\psline{<-|}(6,-3)(8,-3)

\psline{-|}(8,-3)(8.75,-3)
\psline{->}(8.75,-3)(10,-3)

\uput*[270](8.8,-3){$B_1$}
\uput*[270](8,-3){$C_1$}
\rput[bl]{0}(11,-3){$\mathcal{S}_1$}
\rput[bl]{0}(9.1,-2.8){$\mathcal{D}^1_{\sigma_0}$}
\rput[bl]{0}(8.1,-2.8){$\mathcal{D}^1_{\sigma_1}$}

\rput[bl]{0}(6.3,-2.8){$\mathcal{D}^1_{\sigma_3}$}

\psline{<-|}(6,-4.6)(7.25,-4.6)
\psline{-|}(7.25,-4.6)(8,-4.6)
\psline{-|}(8,-4.6)(8.75,-4.6)
\psline{->}(8.75,-4.6)(10,-4.6)
\uput*[270](7.2,-4.6){$A_{\infty}$}
\uput*[270](8.8,-4.6){$B_{\infty}$}
\uput*[270](8,-4.6){$C_{\infty}$}
\rput[bl]{0}(11,-4.6){$\mathcal{S}_{\infty}$}
\rput[bl]{0}(9.1,-4.4){$\mathcal{D}_{\sigma_7}$}
\rput[bl]{0}(8.1,-4.4){$\mathcal{D}_{\sigma_6}$}
\rput[bl]{0}(7.3,-4.4){$\mathcal{D}_{\sigma_5}$}
\rput[bl]{0}(6.3,-4.4){$\mathcal{D}_{\sigma_4}$}

\rput[bl]{0}(6.5,-5.7){\textnormal{deformation of }$\Xhat$}

\end{pspicture}
\caption{Slice decomposition for the deformation induced by $(e_2^*,-e_2)$.}
\end{figure}
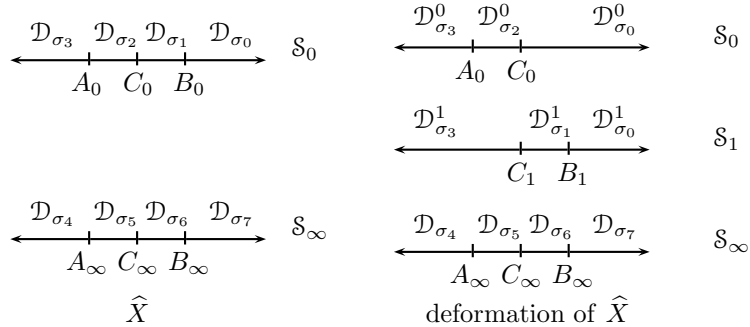
The description of T-invariant divisors from \cite[proposition 3.13. ]{ps} divides these divisors in two types.
Type $1$ divisors are fixed points locus under the $\C^*$-action of $e_1$. Type $2$ divisors
are closure of $\C^*$-orbits and are described by a pair $(Z,v)$
with $Z$ a divisor on $Y=\C\P^1$ and $v$ a vertex of $\mathcal{S}_Z$. On $\Xhat$, type $2$ divisors are:
$$
D_{0,A_0},\; D_{0,B_0},\; D_{0,C_0},\; D_{\infty,A_{\infty}},\; D_{\infty,B_{\infty}}, D_{\infty,C_{\infty}}.
$$
The divisors $D_{0,C_0}$ and $D_{\infty,C_{\infty}}$ are the proper transforms of the fibers of the projection
on the second factor
$$
\C\P^1\times \C\P^1 \rightarrow \C\P^1
$$
on which lie the blown-up points.
The other divisors are the exceptional divisors coming from the $4$ blow-ups of $\C\P^1\times\C\P^1$ that 
lead to $\Xhat$.
Then the divisors of the deformed variety are
$$
D_{0,A_0},\; D_{0,C_0}, \; D_{1,B_1},\; D_{1,C_1},\; D_{\infty,A_{\infty}},\; D_{\infty,B_{\infty}}, \;D_{\infty,C_{\infty}}.
$$
Thus this deformation corresponds to moving the blown-up points corresponding to $D_{0,B_0}$ and $D_{0,A_0}$ on the fixed locus of 
the $\C^*$-action generated by $e_1$, so that they do not lie on the same fiber anymore.
If we consider the polystable deformation generated by $x_2+x_{-2}\in H^1(e^*_2)\oplus H^1(-e_2^*)$,
we obtain projective CSC deformations that preserve a $\C^*$-action. These deformations are described by moving the blown-up points,
 say $P_0$ and $P_{\infty}$,
corresponding to $D_{0,B_0}$ and $ D_{\infty,B_{\infty}}$, on the fixed-set $\C\P^1$ of the $\C^*$-action generated by $e_1$.

Using the symmetry of the situation, the T-deformations that preserves the $\C^*$-action generated by $e_2$ are obtained
by moving the points $P_0$ and $P_{\infty}$ on the fixed locus of this action. Then in the identification
$$
H^1(\Xhat,\tsheaf_ {\Xhat})=H^1(e_1^*)\oplus H^1(-e_1^*) \oplus H^1(e_2^*) \oplus H^1(-e_2^*)\simeq \C^4
$$
the first and third coordinates are identified with the coordinates of $P_0\in\C\P^1\times\C\P^1$ and
the second and the last coordinates correspond to the coordinates of $P_{\infty}\in\C\P^1\times\C\P^1$.

\subsection{Rigid extremal metrics}

Now we want to apply our deformation theory to extremal metrics of non-constant scalar curvature.
We start with $\F_a$ endowed with Calabi's extremal metric in a rational class. Then we consider an action of 
$\Z_p$ on $\F_a$. First, is $\xi$ is a generator of $\mu_p$, the action on $\C\P^1$:
$$
\xi\cdot [u,v] = [\xi u, v]
$$
induces an action on $\mathcal{O}(-1)$ and thus on $\mathcal{O}(-a)$.
Then the action that we consider is the natural extension of this action to $\F_a$.
It acts by isometries and thus we obtain an extremal orbifold $\F_a/\Z_p$.
By a result from \cite{ct}, we know that the minimal resolution using Hirzebruch-Jung strings of this orbifold admits
an extremal metric and we can suppose that this metric defines a polarization. 
Moreover, we can prescribe the $S^1$ action of the extremal vector field on the resolution.
Indeed, the inversion $I$ on the base
$$
I: [u,v]\in \C\P^1 \mapsto [v,u]
$$
lifts to an isometry on $\F_a$ that preserves the $\Z_p$-orbits.
Thus it descends to $\F_a/\Z_p$. Working modulo this inversion, 
we only preserve the $S^1$-action induced by the vertical vector field.
We obtain an extremal metric on a resolution $\Xhat$ of $\F_a/\Z_p$
with a vertical extremal vector field.

In order to apply our deformation theory to this manifold, we need a fan description of
this toric manifold. We will proceed to the description in the case $\F_2/\Z_3$,
even if this discussion extends to the other cases.

$\F_2$ is described by the fan $\Sigma_2$ with:
$$
\Sigma_2^{(1)}=\lbrace e_1,-e_1-2e_,e_2,-e_2 \rbrace,
$$
(Figure 7).
\begin{figure}[htbp]
\psset{unit=0.85cm}
\begin{pspicture}(-4.5,-4.5)(4.5,1)

$$
\xymatrix @M=0mm{
\bullet & \bullet & \bullet & \bullet & \bullet \\
\bullet & \bullet & \bullet & \bullet & \bullet \\
 \bullet & \bullet & \bullet\ar[r]^{e_1}\ar[u]^{e_2}\ar[d]\ar[ldd]& \bullet & \bullet \\
\bullet & \bullet & \bullet & \bullet & \bullet \\
\bullet & \bullet & \bullet & \bullet & \bullet 
}
$$
\end{pspicture}
\caption{$\F_2$}
\end{figure}

and an affine open cover is given by
$$
X_1=Spec(\C[Z,Y]), X_2=Spec(\C[Z,Y^{-1}]),
$$
$$
X_3=Spec(\C[Z^{-1},Z^{-2}Y]) \text{ and } X_4=Spec(\C[Z^{-1},Z^2Y^{-1}]).
$$
corresponding to the four cones
$$
\R^{+*}e_1\oplus\R^{+*}e_2, \R^{+*}e_1\oplus\R^{+*}-e_2,
$$
$$
\R^{+*}-e_1-2e_2\oplus\R^{+*}e_2 \text{ and } \R^{+*}-e_1-2e_2\oplus\R^{+*}-e_2.
$$
Then the action of $\Z_3$ reads
$$
(Z,Y) \mapsto (\xi Z,\xi Y)
$$
so that the fan $\Sigma_s$ of $\F_2/\Z_3$
is given by the coarsest fan with
$$
\Sigma^{(1)}_s=\lbrace e_2, 3e_1- e_2, -e_2, -3e_1-e_2 \rbrace
$$
represented by figure 8.
\begin{figure}[htbp]
\psset{unit=0.85cm}
\begin{pspicture}(-3.5,-4.5)(5,1)
 
$$
\xymatrix @M=0mm{
\bullet & \bullet & \bullet & \bullet & \bullet & \bullet & \bullet\\
\bullet &\bullet & \bullet & \bullet & \bullet & \bullet & \bullet\\
\bullet & \bullet & \bullet & \bullet\ar[u]\ar[rrrd]\ar[d]\ar[llld]& \bullet & \bullet& \bullet \\
\bullet &\bullet & \bullet & \bullet & \bullet & \bullet & \bullet\\
\bullet &\bullet & \bullet & \bullet & \bullet & \bullet & \bullet
}
$$
\end{pspicture}
\caption{$\F_2/\Z_3$}
\end{figure}

Indeed, an open cover for the toric manifold associated to $\Sigma_s$ is
$$
X_{s1}=Spec(\C[U,W,UW^2,UW^3]), X_{s2}=Spec(\C[U,W^{-1},U^{-1}W^{-3}]),
$$
$$
X_{s3}=Spec(\C[U^{-1},W,U^{-1}W^2,U^{-1}W^3]) \text{ and } X_{s4}=Spec(\C[U^{-1},W^{-1},UW^{-3}]).
$$
and the change of variables 
$$
W=ZY^{-1} \text{ and } U=Z^3
$$
shows that if $X_i=Spec(A_i)$, then $X_{si}=Spec(A_i^{\Z_3})$.

Then the toric minimal resolution $\Xhat$ is described by the fan $\Sigma$ with
$$
\Sigma^{(1)}=\lbrace e_1, 3e_1-e_2, 2e_1-e_2, e_1-e_2,-e_2, -e_1, -3e_1-e_2,-2 e_1-e_2, -e_1-e_2,e_2\rbrace .
$$
We represent it Figure 9.

\begin{figure}[htbp]
\psset{unit=0.85cm}
\begin{pspicture}(-3.5,-4.5)(5,1)
 
$$
\xymatrix @M=0mm{
\bullet &\bullet & \bullet & \bullet & \bullet & \bullet & \bullet \\
\bullet &\bullet & \bullet & \bullet & \bullet & \bullet & \bullet\\
\bullet & \bullet & \bullet & \bullet\ar[r]\ar[rrrd]\ar[rrd]\ar[rd]\ar[d]\ar[l]\ar[llld]\ar[lld]\ar[ld]\ar[u]& \bullet & \bullet & \bullet\\
\bullet &\bullet & \bullet & \bullet & \bullet & \bullet& \bullet \\
\bullet &\bullet & \bullet & \bullet & \bullet & \bullet & \bullet
}
$$
\end{pspicture}
\caption{ $\Xhat$}
\end{figure}

We recognize a twice three-times iterated blow-up of $\C\P^1\times \C\P^1$.
Then the vertical action corresponds to the action induced by $e_2$ in $\TT^{\C}=N\otimes_{\Z}\C^*$.
We compute:
$$
H^1(\Xhat,\tsheaf_ {\Xhat})= H^1(e_2^*)\oplus H^1(-e_1^*-e_2^*)\oplus H^1(-e_1^*-2e_2^*)\oplus H^1(e_1^*-e_2^*)\oplus H^1(e_1^*-2e_2^*).
$$
We see that this manifold admits several polystable deformations that preserves $S^1$ actions,
but none that preserves the extremal vector field. Thus $\Xhat$ admits no projective extremal deformation.

However, if we blow-up twice this manifold, working modulo the inversion and using the theorem of 
Arezzo Pacard and Singer \cite{aps}, we obtain 
an extremal metric on the manifold $\Xhat_2$ described by the fan with
$$
\Sigma^{(1)}(2)=\Sigma^{(1)}\cup\lbrace  e_1+e_2, -e_1+e_2\rbrace .
$$
Here, 
$$
H^1(\Xhat_2,\tsheaf_ {\Xhat_2})=  H^1(-e_2^*)\oplus H^1(e_2^*)\oplus H^1(-e_1^*-e_2^*)\oplus H^1(-e_1^*-2e_2^*)
$$
$$
\oplus H^1(e_1^*-e_2^*)\oplus H^1(e_1^*-2e_2^*)\oplus H^1(e_1^*)\oplus H^1(-e_1^*).
$$
Then
$$
H^1(\Xhat_2,\tsheaf_ {\Xhat_2})^{\TT_f}=H^1(e_1^*)\oplus H^1(-e_1^*)
$$
and by theorem~\ref{equivextremal}, $\Xhat_2$ admits projective extremal deformations.

\end{document}